\documentclass[12pt]{article}
\usepackage[titletoc,title]{appendix}
\usepackage{amsmath,amsfonts,amssymb,bm,amsthm}
\usepackage[colorlinks,citecolor=blue,linkcolor=blue]{hyperref}
\usepackage{tikz}
\usetikzlibrary{shapes,arrows}
\usetikzlibrary{intersections,shapes.arrows}
\usetikzlibrary{decorations.pathreplacing}
\usetikzlibrary{arrows.meta}
\numberwithin{equation}{section}
\newtheorem{theorem}{Theorem}[section]
\newtheorem{lemma}[theorem]{Lemma}
\newtheorem{proposition}[theorem]{Proposition}

\theoremstyle{definition}

\newtheorem{definition}[theorem]{Definition}

\theoremstyle{remark}
\newtheorem{remark}[theorem]{Remark}
\newtheorem{notation}[theorem]{Notation}

\usepackage{cases}

\usepackage[T1]{fontenc}
\usepackage[top=1in, bottom=1in, left=1in, right=1in]{geometry}
\usepackage{fancyhdr}
\newcommand{\prin}{\mathrm{prin}}
\newcommand{\sub}{\mathrm{sub}}
\usepackage{enumerate}
\usepackage{xcolor,cancel}
\usepackage{relsize}

\renewcommand{\tilde}{\widetilde}

\allowdisplaybreaks

\begin{document}

\title{Diagonalization of elliptic systems via pseudodifferential projections}
\author{Matteo Capoferri
\thanks{
MC:
School of Mathematics,
Cardiff University,
Senghennydd Rd,
Cardiff CF24~4AG,
UK;\newline
matteo.capoferri@gmail.com,
\url{http://mcapoferri.com}.
}
}

\renewcommand\footnotemark{}

\date{12 January 2022}

\maketitle
\begin{abstract}
Consider an elliptic self-adjoint pseudodifferential operator $A$ acting on $m$-columns of half-densities on a closed manifold $M$,  whose principal symbol is assumed to have simple eigenvalues.  Relying on a basis of pseudodifferential projections commuting with $A$,  we construct an almost-unitary pseudodifferential operator that diagonalizes $A$ modulo an infinitely smoothing operator. We provide an invariant algorithm for the computation of its full symbol, as well as an explicit closed formula for its subprincipal symbol. 
Finally, we give a quantitative description of the relation between the spectrum of $A$ and the spectrum of its approximate diagonalization, and discuss the implications at the level of spectral asymptotics.

\


\

{\bf 2020 MSC classes: }
primary 
58J40;
secondary
47A15,
35J46,
35J47,
35J48,
35P20,
58J05.

\end{abstract}

\tableofcontents

\allowdisplaybreaks

%
%
%
%
%

\section{Introduction}
\label{Introduction}

Decoupling systems of (partial differential) equations is often quite useful, especially in applications, see, e.g., \cite{FH50,application1, applications2,applications3, applications4,applications5}. 
A number of diagonalization techniques have been developed over the years, in different contexts and with varying degree of mathematical rigour. In this paper we are concerned with the diagonalization of elliptic systems of pseudodifferential equations.

In 1975 Taylor \cite{taylor_diag} showed that a first order system of pseudodifferential operators can be approximately diagonalized, as soon as one can diagonalize its principal symbol in a smooth manner.  Taylor's construction establishes the existence of a diagonalizing operator by starting from the principal symbol and iteratively seeking correction terms of lower order. The matrix correction terms are prescribed to have a particular structure, so that the uniqueness of the diagonalizing operator is guaranteed, modulo an infinitely smoothing operator. An analogue of Taylor's argument in the semiclassical setting was later provided by Helffer and Sj\"ostrand \cite{helffer_sj}.

More detailed results are available for specific operators.  Let us mention a few.

In 1983 Cordes \cite{cordes} performed a rigorous analysis of the Foldy--Wouthuysen transform \cite{FH50,foldy},  a unitary operator that achieves the (approximate) decoupling of the positive and negative energy parts of the (massive) Dirac equation, using the tools of microlocal analysis.  Unlike Taylor, Cordes does not introduce artificial constraints in his construction; as a result, the unitary operator he obtains is not uniquely defined --- not even modulo an operator of order $-\infty$. In a subsequent paper \cite{cordes2} Cordes investigated whether one can correct the Foldy--Wouthuysen transform by smoothing operators so as to achieve an exact decoupling,  not merely modulo smoothing.  His analysis identified an obstruction to an exact diagonalization, expressed in terms of the nonvanishing of a certain `deficiency index'.

A semiclassical Foldy--Wouthuysen transform for the Dirac operator was later constructed by Bruneau and Robert \cite{BR99}, following Taylor's strategy.

Nenciu and Sordoni \cite{NS04} studied the approximate diagonalization and the resulting almost-invariant subspaces of multistate Klein--Gordon Hamiltonians.  Their analysis was taken further by Panati, Spoh and Teufel \cite{PST03} who, in the setting of adiabatic perturbation theory,  considered more general Hamiltonians with `internal' degrees of freedom. 

We should mention that diagonalization arguments feature prominently in various areas of the mathematical analysis and mathematical physics literature --- see, e.g., \cite{LF91,blount, LW93,wegner,flow}.  These include approaches to (block) diagonalization that do not rely on pseudodifferential techniques,  see also \cite{tretter,siedentop,cuenin}.

In this paper we will approach diagonalization of elliptic pseudodifferential systems by a different route, one that is based on the theory of pseudodifferential projections, as developed by Vassiliev and the author in \cite{part1}.  Our global invariant construction is different from those listed above, and will lead to an explicit algorithm for the determination of the full symbol of the almost-unitary diagonalizing operator. Being adjusted to pseudodifferential projections,  our construction is particularly suitable to study the relation between the spectrum of the original operator and that of its diagonalized version.
We believe that the self-contained explicit construction presented here will prove useful in applications ---  for example, in the context of quantum field theory in curved spacetime,  compare \cite{bolte, cordes3, brummelhuis} and \cite{gerard1, gerard2, gerard3,shen_wrochna} --- and be more accessible for the wider mathematical physics community.  From a theoretical perspective, this paper complements the analysis of invariant subspaces of elliptic systems carried out in \cite{part1, part2}, and the spectral results therein.

\section{Statement of the problem}
\label{Statement of the problem}

Let $M$ be a closed connected manifold of dimension $d\ge 2$. We denote local coordinates on $M$ by $x=(x^1, \ldots, x^d)$ and local coordinates on the cotangent bundle $T^*M$ by $(x,\xi)=(x^1, \ldots,x^d, \xi_1, \ldots, \xi_d)$. We denote by $T'M:=T^*M\setminus \{0\}$ the cotangent bundle with the zero section removed.

Let $\mathcal{C}^\infty(M)$ be the vector space of smooth complex-valued half-densities over $M$ and let $\mathcal{L}^2(M)$ be its closure with respect to the inner product
\begin{equation*}
\prec\! f,g\! \succ=\int_M \overline{f}\,g\, dx,
\end{equation*}
where $dx:=dx^1\dots dx^n$ and the overbar denotes complex conjugation.

Given a natural number $m\ge 2$,  we define $C^\infty(M):=[\mathcal{C}^\infty(M)]^m$ to be the vector space of $m$-columns of smooth complex-valued half-densities over $M$ and $L^2(M)$ to be its closure with respect to the inner product
\begin{equation*}
\langle u,v \rangle=\int_M u^*v \,dx.
\end{equation*}
Here and further on $*$ stands for Hermitian conjugation when applied to matrices and for adjunction when applied to operators. Of course, $L^2(M):=[\mathcal{L}^2(M)]^m$.  The spaces $C^\infty(M)$ and $L^2(M)$ depend on the choice of $m$, but we suppress this dependence to simplify notation.

For $s\in \mathbb{R}$, we denote by $\mathcal{H}^s(M)$ (resp.~$H^s(M)$) the usual Sobolev space,  i.e.~the vector space of half-densities (resp.  $m$-columns of half-densities) that are in $\mathcal{L}^2(M)$ (resp.~$L^2(M)$) together with their partial derivatives up to order $s$ for $s\ge0$,  and defined by duality for negative $s$.

Let $\Psi^s_{k,n}$ be the space of classical pseudodifferential operators of order $s$ with polyhomogeneous symbol, acting from $[\mathcal{H}^s(M)]^k$ to $[\mathcal{L}^2(M)]^{n}$. In other words,  operators in $\Psi^s_{k,n}$ map $k$-columns to $n$-columns.  We will adopt the notation $\Psi^s_k:=\Psi^s_{k,k}$.  For an operator $Q\in \Psi^s_{k,n}$ we denote by $Q_\prin$ and $Q_\sub$ its principal and subprincipal symbols \cite[Proposition~5.2.1]{DuHo},  respectively.

We denote by $\mathrm{Id}_m\in \Psi^0_m$ the identity operator on $L^2(M)$, by $\mathrm{Id}\in \Psi^0_1$ the identity operator on $\mathcal{L}^2(M)$,  and by $I$ the $m\times m$ identity matrix.

Let $A\in \Psi^s_m$,  $s>0$,  be an elliptic self-adjoint linear pseudodifferential operator, where ellipticity means that
\begin{equation*}
\det A_\prin(x,\xi)\ne 0 \qquad \forall (x,\xi)\in T'M.
\end{equation*}
Throughout this paper we assume that $A_\prin$ has simple eigenvalues.  Let $m^+$ (resp.~$m^-$) be the number of positive (resp.~negative) eigenvalues of $A_\prin$.  We denote by $h^{(j)}(x,\xi)$ the eigenvalues of $A_\prin(x,\xi)$, enumerated in increasing order with positive index $j=1,\ldots,m^+$ for positive $h^{(j)}(x,\xi)$ and negative index $j=-1,\ldots,-m^-$ for negative $h^{(j)}(x,\xi)$.
We denote by
\begin{equation*}
J:=[-m^-,m^+]\cap(\mathbb{Z}\setminus\{0\})
\end{equation*}
the indexing set of admissible values of $j$.

\begin{remark}
Each $h^{(j)}$ is a nowhere-vanishing smooth real-valued function on the punctured cotangent bundle $T'M$. Furthermore, we have that $m^+$ and $m^-$ are constant and $m=m^++m^-$. These facts are immediate consequences of self-adjointness, ellipticity,  and connectedness of $M$.
\end{remark}

By $v^{(j)}(x,\xi)$ we denote a normalised eigenvector of $A_\prin(x,\xi)$ associated with the eigenvalue $h^{(j)}(x,\xi)$, and by 
\begin{equation*}
P^{(j)}(x,\xi):=v^{(j)}(x,\xi)[v^{(j)}(x,\xi)]^*
\end{equation*}
the corresponding eigenprojection.  Throughout this paper, we assume that one can choose eigenvectors $v^{(j)}(x,\xi)$,  $j\in J$,  of $A_\mathrm{prin}$ smoothly for all $(x,\xi)\in T^*M\setminus\{0\}$.

\begin{remark}
\label{remark topological obstructions}
In general, there may exist topological obstructions preventing one from being able to choose eigenvectors of the principal symbol of a matrix pseudodifferential operator in a smooth manner globally on the whole cotangent bundle, or even in the whole cotangent fibre at a single point.  In the current paper, we assume that there are no such obstructions, so that one can perform a global diagonalization.  Necessary and sufficient conditions for the existence of topological obstructions will be examined extensively in a separate paper \cite{obstructions}.
\end{remark}

Note that,  for each $j$, $P^{(j)}(x,\xi)$ is a uniquely defined rank $1$ matrix function,  whereas $v^{(j)}(x,\xi)$ is only defined up to a local gauge transformation
\begin{equation}
\label{gauge transformation}
v^{(j)}\mapsto e^{i\phi^{(j)}}v^{(j)},
\end{equation}
where $\phi^{(j)}:T'M \to \mathbb{R}$ is an arbitrary smooth function.
The eigenprojections satisfy
\begin{equation*}
A_\prin=\sum_{j\in J} h^{(j)}P^{(j)}, \qquad \sum_{j\in J} P^{(j)}=I.
\end{equation*}

The spectrum of the operator $A:H^s(M)\to L^2(M)$ is discrete and it accumulates to infinity. More precisely,  if $m^+\ge1$ it accumulates to $+\infty$, if $m^-\ge 1$ it accumulates to $-\infty$.

\begin{definition}
\label{definition of sign definiteness modulo}
We say that a symmetric pseudodifferential operator $Q\in \Psi^s_m$ is \emph{nonnegative (resp.~nonpositive) modulo} $\Psi^{-\infty}_m$
and write
\[
Q\ge0\mod\Psi^{-\infty}_m \qquad(\text{resp.}\ Q\le0\mod\Psi^{-\infty}_m)
\]
if there exists a symmetric operator $R\in\Psi^{-\infty}_m$
such that $Q+R\ge0$ (resp.~$Q+R\le0$).
\end{definition}

In \cite{part1} the author and Vassiliev proved the following results.

\begin{theorem}
\label{theorem results from part 1}
Let $A$ be as above and let $\delta$ be the Kronecker symbol.

\begin{enumerate}[(a)]

\item 
\cite[Theorem~2.2]{part1} There exist $m$ pseudodifferential operators $P_j\in \Psi^0_m$ satisfying
\begin{enumerate}[(i)]
\item 
$(P_j)_\mathrm{prin}=P^{(j)}$,

\item
$P_j=P_j^* \mod \Psi^{-\infty}_m$,

\item
$P_j P_l =\delta_{jl}\, P_j \mod \Psi^{-\infty}_m$,

\item
$\sum_{j\in J} P_j=\mathrm{Id}_m \mod \Psi^{-\infty}_m$,

\item
$[A,P_j]=0 \mod \Psi^{-\infty}_m$.
\end{enumerate}
These operators are uniquely determined, modulo $\Psi^{-\infty}_m$, by $A$.

\item
\cite[Theorem~2.5]{part1}
The operators $P_j$ from part (a) satisfy
\begin{align*}
&P_j^*AP_j\ge0\mod\Psi^{-\infty}_m\quad\text{for}\quad j=1,\dots,m^+,
\\
&P_j^*AP_j\le0\mod\Psi^{-\infty}_m\quad\text{for}\quad j=-1,\dots,-m^-.
\end{align*}
\end{enumerate}
\end{theorem}

The operators $P_j$ from the above theorem form an orthonormal basis of pseudodifferential projections commuting with $A$,  and an explicit algorithm for the construction of their full symbols was given in \cite{part1}.  Pseudodifferential projections were exploited in \cite{part2} to partition the spectrum of $A$ into precisely $m$ infinite series of eigenvalues,  singling out the contribution to the spectrum of $A$ of individual eigenvalues of $A_\mathrm{prin}$ (more details about this will be recalled later on).  Theorem~\ref{theorem results from part 1} tells us that pseudodifferential projections allow one to decompose $A$ into $m$ sign-semidefinite operators $P_j^*AP_j\in \Psi^s_m$. However, these operators are not elliptic, and this was one of the major obstacles one had to overcome in \cite{part2}.

\

We take here a different perspective on the problem and ask the following questions.

\

{\bf Question 1\ } 
Does there exist a pseudodifferential operator $B\in \Psi^{0}_m$ such that
\begin{equation}
\label{B almost unitary}
B^*B=\operatorname{Id}_m \mod \Psi^{-\infty}_m, \qquad BB^*=\operatorname{Id}_m \mod \Psi^{-\infty}_m,
\end{equation}
and
\begin{equation}
\label{B diagonalises}
B^*AB=\tilde A \mod \Psi^{-\infty}_m,
\end{equation}
where $\tilde A\in \Psi^s_m$ is a diagonal operator?

\

{\bf Question 2\ }
Assuming the answer to Question 1 is affirmative, can one exploit pseudo\-differential projections to construct the operator $B$ explicitly?

\

A positive answer to Question~2 would yield an explicit characterisation of the diagonal operator $\tilde A$.  In particular, it is easy to see that the operators on the diagonal of $\tilde A$ would be `scalar' elliptic operators $a_j\in \Psi^s_1$ with $(a_j)_\prin=h^{(j)}$.  This leads to a further question.

\

{\bf Question 3\ } What is the relation between the spectrum of $A$ and the spectra of $a_j$, $j\in J$?

\

The goal of this paper is to provide a detailed answer to Questions~1, 2 and 3 and discuss some applications.

\section{Main results}
\label{Main results}

Our main results are summarised in this section in the form of five theorems.

\begin{theorem}
\label{main theorem 1}
For each $j\in J$ there exists a pseudodifferential operator $B_j\in \Psi^0_{1,m}$ such that
\begin{equation}
\label{main theorem 1 equation 1}
(B_j)_\prin=v^{(j)},
\end{equation}
\begin{equation}
\label{main theorem 1 equation 2}
B_j^* B_j=\mathrm{Id} \mod \Psi^{-\infty}_1,
\end{equation}
\begin{equation}
\label{main theorem 1 equation 3}
P_j B_j=B_j \mod \Psi^{-\infty}_{1,m}.
\end{equation}
\end{theorem}

An explicit algorithm for the calculation of the full symbol of the operators $B_j$ is given in subsection~\ref{The algorithm}.  Such algorithm shows that the operators $B_j$ are not uniquely determined by conditions \eqref{main theorem 1 equation 1}--\eqref{main theorem 1 equation 3} and characterises the degrees of freedom left in the symbol.

By carrying out the first iteration of the algorithm one obtains the following.
\begin{theorem}
\label{main theorem 2}
The subprincipal symbol of $B_j$ reads
\begin{equation}
\label{main theorem 2 equation 1}
(B_j)_\sub=\left(\frac14\operatorname{tr} ((P_j)_\mathrm{sub})+i\,f^{(j)}\right)v^{(j)}+(P_j)_\sub v^{(j)}+\frac i{2}\{P^{(j)},v^{(j)}\},
\end{equation}
where $f^{(j)}:T'M \to \mathbb{R}$ is an arbitrary smooth real-valued scalar function.
\end{theorem}

The curly brackets appearing in \eqref{main theorem 2 equation 1} denote the Poisson bracket
\begin{equation*}
\label{poisson brackets}
\{B,C\}:=\sum_{\alpha=1}^d(B_{x^\alpha} C_{\xi_\alpha}- B_{\xi_\alpha} C_{x^\alpha})
\end{equation*}
on matrix-functions on the cotangent bundle.
Further on in the paper we will also make use
of the generalised Poisson bracket
\begin{equation*}
\label{generalised poisson brackets}
\{ B,C,D\}:=\sum_{\alpha=1}^d(B_{x^\alpha} C D_{\xi_\alpha}- B_{\xi_\alpha} C D_{x^\alpha}).
\end{equation*}

The operators $B_j$ constitute the building blocks --- more precisely, the `columns' --- of the sought-after almost-unitary operator $B$.
\begin{theorem}
\label{main theorem 3}
The operator $B \in \Psi^{0}_m$ defined in accordance with
\begin{equation}
\label{main theorem 3 equation 1}
B: 
u=\begin{pmatrix}
u_{m^+}\\
\vdots\\
u_1\\
u_{-1}\\
\vdots\\
u_{-m^-}
\end{pmatrix}
\mapsto
\sum_j B_j u_j
\end{equation}
satisfies conditions \eqref{B almost unitary} and \eqref{B diagonalises}. The diagonalised operator $\tilde A$ appearing in \eqref{B diagonalises} reads
\begin{equation}
\label{main theorem 3 equation 2}
\tilde{A}=\operatorname{diag}(a_{m^+}, \ldots, a_1,a_{-1}, \ldots, a_{-m^-}),
\end{equation}
where
\begin{equation}
\label{main theorem 3 equation 3}
a_j:=B_j^* A B_j \in \Psi^{s}_1\,.
\end{equation}
\end{theorem}
Of course, Theorem~\ref{main theorem 3} follows easily from Theorem~\ref{main theorem 1}. The full symbols of both $B$ and the $a_j$'s are explicitly determined by the full symbols of the $B_j$'s.

\

The next two theorems establish a relation between the spectrum of our original operator $A$ and that of the $a_j$, $j\in J$, and show that the diagonalization procedure does not alter the spectrum substantially.

For a self-adjoint operator $Q$, we shall denote its spectrum by $\sigma(Q)$ and by 
\begin{equation*}
\sigma^+(Q):=\sigma(Q) \cap (0,+\infty)
\end{equation*}
its positive spectrum.

Let 
\begin{equation}
\label{positive spectrum A}
0<\lambda_1 \le \lambda_2 \le \ldots\le \lambda_k\le \ldots \to +\infty
\end{equation}
be the positive eigenvalues of $A$, enumerated with account of multiplicity.  Suppose $m^+\ge 1$ and let 
\begin{equation}
\label{positive spectrum a_j}
0<\ell^{(j)}_1 \le \ell^{(j)}_2 \le \ldots\le \ell^{(j)}_k\le \ldots \to +\infty
\end{equation}
be the positive eigenvalues of $a_j$, $j=1,\ldots, m^+$, also enumerated with account of multiplicity.

\begin{theorem}
\label{main theorem 4}

\begin{enumerate}[(a)]

\item We have
\begin{equation}
\label{main theorem 4 equation 1}
\operatorname{dist}(\lambda_k, \textstyle\bigcup_{j=1}^{m^+} \sigma^+(a_j))=O(k^{-\infty}) \quad \text{as}\quad k\to+\infty.
\end{equation}

\item
For each $j=1,\ldots, m^+$ we have
\begin{equation}
\label{main theorem 4 equation 2}
\operatorname{dist}(\ell^{(j)}_k, \sigma^+(A))=O(k^{-\infty})\quad \text{as}\quad k\to+\infty.
\end{equation}

\end{enumerate}
\end{theorem}

Let us combine the sequences of eigenvalues \eqref{positive spectrum a_j} for $j=1,\ldots,m^+$ into a single sequence
\begin{equation}
\label{spectra of a_j combined}
0<\zeta_1\le \zeta_2 \le \ldots \le \zeta_k\le \ldots \to +\infty,
\end{equation}
with account of multiplicity.  The sequence \eqref{spectra of a_j combined} coincides, up to a finite number of eigenvalues, with the positive spectrum of the operator \eqref{main theorem 3 equation 2}.

\begin{theorem}
\label{main theorem 5}
For every $\alpha>0$ there exists $z_\alpha\in \mathbb{Z}$ such that
\begin{equation}
\label{main theorem 5 equation 1}
\lambda_k=\zeta_{k+z_\alpha}+O(k^{-\alpha})\quad \text{as}\quad k\to +\infty.
\end{equation}
\end{theorem}

Theorem~\ref{main theorem 4} establishes asymptotic closeness of positive spectra of $A$ and $a_j$, $j=1,\ldots, m^+$, whereas Theorem~\ref{main theorem 5} establishes asymptotic closeness of individual eigenvalues, enumerated in our particular way.  Note that the remainder terms in our two theorems are different. The underlying reason for this is that when proving one-to-one correspondence between eigenvalues in the limit for $k\to+\infty$ we have to take into account the potential presence of \emph{spectral clusters}, i.e.  eigenspaces of large multiplicity. 

Theorems~\ref{main theorem 4} and \ref{main theorem 5},  alongside their counterpart for $j=-1,\ldots,-m^-$, imply a result first proved in \cite{part2}, namely that the spectrum of $A$ partitions, up to an error with superpolynomial decay,  into $m$ distinct series, one for each eigenvalue of $A_\prin$. The advantage of the approach adopted in the current paper is that these series are characterised as the spectra of $m$ \emph{semibounded scalar} (as opposed to non-semibounded matrix) elliptic operators, the operators $a_j$.  In particular, Theorem~\ref{main theorem 5} implies that one can compute the (global) Weyl coefficients of $A$ as the sum of the (global) Weyl coefficients of the $a_j$'s, hence relying only on the much simpler formulae for \emph{scalar} elliptic operators --- see subsection~\ref{The second Weyl coefficient}. 

\begin{remark}
\label{remark generalisations}
Let us briefly discuss possible generalisations of our results. To begin with,  we should point out that the ellipticity of $A$ or the fact that $A$ is of positive order are not really needed in the proofs of Theorems~\ref{main theorem 1} and~\ref{main theorem 2}, or for the construction algorithm contained in subsection~\ref{The algorithm}.  The key assumption is the simplicity of the eigenvalues of $A_\mathrm{prin}$.  Moreover,  our algorithm can be applied, with minimal modifications, to operators acting on more general vector bundles, e.g.,  to operators acting on differential forms.  We do not carry out these generalisations in the current paper because doing so would hinder clarity and readability,  without bringing a corresponding benefit in terms of insight. Finally, we should mention that in the presence of topological obstructions discussed in Remark~\ref{remark topological obstructions} one can still pursue a local (or microlocal) diagonalization of the operator $A$ with the techniques presented in this paper. However, a diagonalization that is only local is of little or no use in spectral-theoretic applications, as will become clear in Section~\ref{The diagonalized operator}. 
\end{remark}

\

The paper is structured as follows. 

Section~\ref{Construction of the operators $B_j$} is devoted to the study of the operators $B_j$ appearing in Theorem~\ref{main theorem 1}.  In subsection~\ref{Proof of Theorem 1} we prove Theorem~\ref{main theorem 1}; in subsection~\ref{The algorithm} we provide a self-contained explicit algorithm for the computation of the full symbol of $B_j$; in subsection~\ref{The subprincipal symbol} we carry out the first iteration of our algorithm and obtain a closed formula for the subprincipal symbol of $B_j$, thus proving Theorem~\ref{main theorem 2}.  We conclude the section by demonstrating,  in subsection~\ref{Recovering pseudodifferential projections from the operators $B_j$},  that the pseudodifferential projection $P_j$ can be recovered from $B_j$.

Section~\ref{The diagonalized operator} is concerned with the spectral analysis of the approximate diagonalization $\tilde A$ of $A$. After briefly discussing in subsection~\ref{The almost-unitary operator $B$} some basic properties of $\tilde A$,  in subsection~\ref{Spectral analysis} we provide our main spectral argument and prove Theorems~\ref{main theorem 4} and~\ref{main theorem 5}.  Finally, in subsection~\ref{The second Weyl coefficient} we discuss the relation between local and global spectral asymptotics of $A$ and $\tilde A$,  for the special case of first order operators.

In Section~\ref{An example: the operator of linear elasticity} we apply our results to an explicit example: the operator of linear elasticity in dimension 2.  

The paper is complemented by an appendix collecting some auxiliary results.

\section{Construction of the operators $B_j$}
\label{Construction of the operators $B_j$}

This section is devoted to the proof of Theorems~\ref{main theorem 1} and~\ref{main theorem 2}.

Throughout this section we will need a refined notation for the principal symbol. Namely, we denote by $(\,\cdot\,)_{\mathrm{prin},r}$ the principal symbol of the expression within brackets, regarded as an operator in $\Psi^{-r}_{k,n}$.  The relevant values of $k$ and $n$ will be clear from the context, so we suppress the dependence on them in the notation of the refined principal symbol.  We will use the refined notation whenever there is risk of confusion.

\subsection{Proof of Theorem~\ref{main theorem 1}}
\label{Proof of Theorem 1}

For the reader's convenience, let us recall that we want to show that for each $j\in J$ there exists a pseudodifferential operator $B_j\in \Psi^0_{1,m}$ such that
\begin{equation*}
\label{main theorem 1 equation 1bis}
(B_j)_\prin=v^{(j)},
\end{equation*}
\begin{equation*}
\label{main theorem 1 equation 2bis}
B_j^* B_j=\mathrm{Id} \mod \Psi^{-\infty}_1,
\end{equation*}
\begin{equation*}
\label{main theorem 1 equation 3bis}
P_j B_j=B_j \mod \Psi^{-\infty}_{1,m}.
\end{equation*}

\begin{proof}
We will prove the theorem by constructing a sequence of pseudodifferential operators $B_{j,k}\in \Psi^0_{1,m}$, $k=0,1,2,\ldots$,  such that
\begin{equation}
\label{proof of theorem 1 equation 1}
B_{j,k+1}-B_{j,k}\in \Psi^{-k-1}_{1,m},
\end{equation}
\begin{equation}
\label{proof of theorem 1 equation 2}
(B_{j,k})_\prin=v^{(j)},
\end{equation}
\begin{equation}
\label{proof of theorem 1 equation 3}
B_{j,k}^*B_{j,k}=\operatorname{Id} \mod \Psi^{-k-1}_{1,m},
\end{equation}
\begin{equation}
\label{proof of theorem 1 equation 4}
P_jB_{j,k}=B_{j,k} \mod \Psi^{-k-1}_{1,m}.
\end{equation}
More precisely,  for each $j$ we choose an arbitrary $B_{j,0}$ satisfying $(B_{j,0})_\prin=v^{(j)}$ and determine recursively pseudodifferential operators $Q_{j,k}\in \Psi^{-k}_{1,m}$, $k=1,2,\ldots$, such that
\begin{equation}
\label{proof of theorem 1 equation 5}
B_{j,k}:=B_{j,k-1}+Q_{j,k}, \qquad k=1,2,\ldots
\end{equation}
satisfy \eqref{proof of theorem 1 equation 3}--\eqref{proof of theorem 1 equation 4}.  Note that the operators $B_{j,k}$ defined in accordance with \eqref{proof of theorem 1 equation 5} automatically satisfy conditions \eqref{proof of theorem 1 equation 1} and \eqref{proof of theorem 1 equation 2}.

Satisfying \eqref{proof of theorem 1 equation 3} and \eqref{proof of theorem 1 equation 4} reduces to solving
\begin{equation*}
\label{proof of theorem 1 equation 6*}
\left((B_{j,k-1}+Q_{j,k})^*(B_{j,k-1}+Q_{j,k})-\operatorname{Id}\right)_{\prin,k}=0\,
\end{equation*}
and
\begin{equation*}
\label{proof of theorem 1 equation 7}
\left(P_j(B_{j,k-1}+Q_{j,k})-B_{j,k-1}-Q_{j,k}\right)_{\prin,k}=0\,,
\end{equation*}
which translate into the following system of equations for $(Q_{j,k})_\prin$:
\begin{subequations}
\begin{numcases}{}
\label{proof of theorem 1 equation 8}
[v^{(j)}]^*(Q_{j,k})_\prin+[(Q_{j,k})_\prin]^*v^{(j)}=r_{j,k}\,,
\\
\label{proof of theorem 1 equation 8bis}
(P^{(j)}-I)(Q_{j,k})_\prin =R_{j,k}\,,
\end{numcases}
\end{subequations}
with
\begin{equation}
\label{proof of theorem 1 equation 9}
r_{j,k}:=(\operatorname{Id}-B_{j,k-1}^*B_{j,k-1})_\mathrm{prin,k}\,,
\qquad
R_{j,k}:=(B_{j,k-1}-P_jB_{j,k-1})_{\prin,k}\,.
\end{equation}

It is not hard to see that the solvability condition $P^{(j)}R_{j,k}=0$ for \eqref{proof of theorem 1 equation 8bis} is satisfied. Indeed, we have
\begin{equation*}
\begin{split}
\label{proof of theorem 1 equation 10}
P^{(j)}R_{j,k}
&
=
(P_j)_{\prin, 0}(B_{j,k-1}-P_jB_{j,k-1})_{\prin,k}
\\
&
=
(P_jB_{j,k-1}-P_jB_{j,k-1})_{\prin,k}
\\
&
=0.
\end{split}
\end{equation*}
In the above calculation we used the fact that the principal symbol map is a homomorphism.

Then
\begin{equation*}
\label{proof of theorem 1 equation 11}
\begin{split}
(Q_{j,k})_\prin
&
=
\left(\frac12 r_{j,k}+ i f_{j,k} \right)v^{(j)}- R_{j,k}
\\
&
=
\left(\frac12 r_{j,k}+ i f_{j,k} \right)v^{(j)}- \sum_{l\in J\setminus \{j\}} \left([v^{(l)}]^*R_{j,k}\right) \, v^{(l)}\,,
\end{split}
\end{equation*}
where $f_{j,k}:T'M \to \mathbb{R}$ is an arbitrary smooth real valued function,  solves \eqref{proof of theorem 1 equation 8}--\eqref{proof of theorem 1 equation 9}.
\end{proof}

\subsection{The algorithm}
\label{The algorithm}

We reformulate the results from the previous subsection in the form of a self-contained algorithm for the construction of the full symbol of the operators $B_j$ and, as a straightforward consequence of \eqref{main theorem 3 equation 1}, of the almost-unitary operator $B$.

\

{\bf Step 1.\ } Given $m$ orthonormal eigenvectors $v^{(j)}$ of $A_\prin$,  choose $m$ arbitrary pseudodifferential operators $B_{j,0}\in\Psi^0_{1,m}$ satisfying $(B_{j,0})_\prin=v^{(j)}$.

\

{\bf Step 2.\ } For $k=1,2,\ldots$ define
\begin{equation}
B_{j,k}:=B_{j,0}+\sum_{n=1}^kQ_{j,n}, \qquad Q_{j,n}\in \Psi^{-n}_{1,m}.
\end{equation}

Assuming we have determined the pseudodifferential operator $B_{j,k-1}$, compute the following quantities:
\begin{enumerate}[(i)]
\item $r_{j,k}:=(\operatorname{Id} -B_{j,k-1}^*B_{j,k-1})_{\prin, k}$,

\item $R_{j,k}:=(B_{j,k-1}-P_jB_{j,k-1})_{\prin,k}$.
\end{enumerate}

\

{\bf Step 3.\ } Choose a pseudodifferential operator $Q_{j,k} \in \Psi^{-k}_{1,m}$ satisfying
\begin{equation}
\label{Qjk prin}
(Q_{j,k})_\prin=\left(\tfrac12 r_{j,k}+i f_{j,k}\right)v^{(j)}-R_{j,k},
\end{equation}
where $f_{j,k}:T'M \to \mathbb{R}$ is an arbitrary smooth real-valued function.

\

{\bf Step 4.\ } Put
\begin{equation}
B_j\sim B_{j,0}+\sum_{k=1}^{+\infty} Q_{j,k}.
\end{equation}
Here $\sim$ stands for asymptotic expansion in smoothness, namely,
\[
B_j- \left(B_{j,0}+\sum_{k=1}^{n} Q_{j,k}\right)\in \Psi^{-n-1}_{1,m}.
\]

\begin{remark}
It transpires from the above algorithm that there are two different sources for the non-uniqueness of $B$:
\begin{enumerate}[(i)]
\item the eigenvectors $v^{(j)}$ are only defined up to a gauge transformation \eqref{gauge transformation};

\item once we have fixed the eigenvectors $v^{(j)}$, for each $B_j$ at every stage of the iterative process we have one real-valued degree of freedom, see \eqref{Qjk prin}.
\end{enumerate}
This agrees with what observed, in a somewhat different setting, by Panati, Spohn and Teufel \cite{PST03}. We should mention that great care is needed when dealing with the gauge transformation \eqref{gauge transformation}: failing to properly account for \eqref{gauge transformation} and the curvature it brings about led to mistakes in spectral-theoretic publications, see \cite[Sec.~11]{CDV}.  See also \cite{nicoll, spin1, sesqui,leonid} for a wider picture on the role of gauge transformations in the analysis of systems of PDEs.

\end{remark}

\subsection{The subprincipal symbol}
\label{The subprincipal symbol}

Let us carry out the first iteration of our algorithm to obtain a formula for the subprincipal symbol of $B_j$, thus proving Theorem~\ref{main theorem 2}.

\

Without loss of generality, let us choose an operator $B_{j,0}\in \Psi^0_{1,m}$ satisfying $(B_{j,0})_\prin=v^{(j)}$ and the additional property
\begin{equation}
\label{extra condition Bj0 sub equal zero}
(B_{j,0})_\sub=0.
\end{equation}

Recall that the subprincipal symbol of the composition of two pseudodifferential operators is given by
\begin{equation}
\label{subprincipal symbol of the product}
(CD)_\sub=C_\prin D_\sub+C_\sub D_\prin+\frac{i}2\{C_\prin, D_\prin\},
\end{equation}
see \cite[Eqn~(1.4)]{DuGu} where the opposite convention for the sign of the Poisson brackets is adopted.

On account of \eqref{extra condition Bj0 sub equal zero}, \eqref{subprincipal symbol of the product}, \eqref{main theorem 1 equation 1} and the fact that $\operatorname{Id}_\sub=0$, we have
\begin{equation}
\label{computing subprincipal symbol Bj 1}
\begin{split}
r_{j,1}
&
=
(\operatorname{Id} -B_{j,0}^*B_{j,0})_\sub
\\
&
=
-\left([v^{(j)}]^*(B_{j,0})_\sub+[(B_{j,0})_\sub]^*v^{(j)}+\frac{i}2\{[v^{(j)}]^*,v^{(j)}\}\right)
\\
&
=
\frac12 \operatorname{tr}((P_j)_\sub).
\end{split}
\end{equation}
In the last step of the above calculation we resorted to the identity
\begin{equation*}
\{[v^{(j)}]^*,v^{(j)}\}=i\operatorname{tr}((P_j)_\sub),
\end{equation*}
which is obtained by taking the trace in \cite[Theorem~2.3]{part1} or by combining \cite[Eqn.~(1.20)]{CDV} with \cite[Corollary~4.2]{part2}. 

Similarly,  in view of \eqref{extra condition Bj0 sub equal zero} and \eqref{subprincipal symbol of the product}
we get
\begin{equation}
\label{computing subprincipal symbol Bj 2}
\begin{split}
R_{j,1}
&
=
(B_{j,0}-P_jB_{j,0})_\sub  
\\
&
=
-\left(P^{(j)}(B_{j,0})_\sub+(P_j)_\sub v^{(j)}+\frac{i}2\{P^{(j)},v^{(j)}\}\right)
\\
&
=
-\left((P_j)_\sub v^{(j)}+\frac{i}2\{P^{(j)},v^{(j)}\}\right).
\end{split}
\end{equation}

Substituting \eqref{computing subprincipal symbol Bj 1} and \eqref{computing subprincipal symbol Bj 2} into \eqref{Qjk prin} we arrive at \eqref{main theorem 2 equation 1}.

\begin{remark}
It is not hard to check directly from \eqref{computing subprincipal symbol Bj 2} that $[v^{(j)}]^*R_{j,1}=0$.  Indeed, \cite[Theorem~2.3]{part1} tells us that
\begin{equation}
\label{remark on checking solvability condition 1}
[v^{(j)}]^*(P_j)_\sub v^{(j)}=-\frac{i}2 [v^{(j)}]^*\{P^{(j)},P^{(j)}\} v^{(j)},
\end{equation}
whereas a direct calculation exploiting properties of Poisson brackets gives us
\begin{equation*}
\label{remark on checking solvability condition 2}
\begin{split}
\frac{i}2[v^{(j)}]^*\{P^{(j)},v^{(j)}\}
&
=
\frac{i}2[v^{(j)}]^*\{P^{(j)},P^{(j)}v^{(j)}\}\\
&
=\frac{i}2[v^{(j)}]^*\left(\{P^{(j)},P^{(j)},v^{(j)}\}+\{P^{(j)},P^{(j)}\}v^{(j)}\right)
\\
&
=\frac{i}2[v^{(j)}]^*\left(\{(P^{(j)})^2,v^{(j)}\}-P^{(j)}\{P^{(j)},v^{(j)}\}+\{P^{(j)},P^{(j)}\}v^{(j)}\right)
\\
&
=
\frac{i}2[v^{(j)}]^*\{P^{(j)},P^{(j)}\}v^{(j)},
\end{split}
\end{equation*}
which is the same as \eqref{remark on checking solvability condition 1}, but with opposite sign.
\end{remark}

\subsection{Recovering $P_j$ from $B_j$}
\label{Recovering pseudodifferential projections from the operators $B_j$}


Before moving on to the spectral analysis,  let us prove an identity relating the operators $B_j$ and pseudodifferential projections $P_j$, which will be quite useful later on.

\begin{proposition}
\label{proposition on Bj vs Pj}
We have
\begin{equation}
\label{proposition on Bj vs Pj equation 1}
B_j B_j^*=P_j \mod \Psi^{-\infty}_m, \qquad j\in J,
\end{equation}
and
\begin{equation}
\label{proposition on Bj vs Pj equation 2}
\sum_{j\in J} B_j B_j^*=\operatorname{Id}_m \mod \Psi^{-\infty}_m\,.
\end{equation}
\end{proposition}

\begin{proof}
Let us define
\begin{equation*}
\label{proof proposition on Bj vs Pj equation 1}
S_j:=B_j B_j^*-P_j.
\end{equation*}
Arguing by contradiction, suppose that
\begin{equation*}
\label{proof proposition on Bj vs Pj equation 2}
S_j\in \Psi^{-k}_m \quad \text{but}\quad S_j\not\in \Psi^{-k-1}_m.
\end{equation*}
This means that
\begin{equation}
\label{proof proposition on Bj vs Pj equation 3}
(S_j)_{\prin,k}(x,\xi)\ne 0 \quad \text{for some}\quad (x,\xi)\in T'M.
\end{equation}

Now,  formula \eqref{main theorem 1 equation 3} and Theorem~\ref{theorem results from part 1}(a) imply
\begin{equation*}
\label{proof proposition on Bj vs Pj equation 4}
S_j \sum_{l\in J}P_l =0 \mod \Psi^{-\infty}_m.
\end{equation*}
The latter, in particular, implies
\begin{equation}
\label{proof proposition on Bj vs Pj equation 5}
\left(S_j \sum_{l\in J}P_l\right)_{\prin,k}=0.
\end{equation}
But
\begin{equation}
\label{proof proposition on Bj vs Pj equation 6}
\left(S_j \sum_{l\in J}P_l\right)_{\prin,k}=(S_j)_{\prin,k}\left(\sum_{l\in J}P_l\right)_{\prin,0}=(S_j)_{\prin,k}.
\end{equation}
Formulae \eqref{proof proposition on Bj vs Pj equation 5} and \eqref{proof proposition on Bj vs Pj equation 6} contradict \eqref{proof proposition on Bj vs Pj equation 3}, so we have proved \eqref{proposition on Bj vs Pj equation 1}.

Formula \eqref{proposition on Bj vs Pj equation 2} now follows from \eqref{proposition on Bj vs Pj equation 1} and Theorem~\ref{theorem results from part 1}.
\end{proof}

\section{The diagonalized operator}
\label{The diagonalized operator}

\subsection{Action of the almost-unitary operator $B$}
\label{The almost-unitary operator $B$}

Define the matrix operator
\begin{equation}
\label{definition of B as matrix operator}
B:=
\begin{pmatrix}
B_{m^+}, \ldots, B_1, B_{-1}, \ldots, B_{m^{-1}}
\end{pmatrix}.
\end{equation}
Of course, $B\in \Psi^{0}_m$ and its action on $L^2(M)$ is given by \eqref{main theorem 3 equation 1}. 
Let us label rows and columns of matrix operators from $\Psi^r_m$, $r\in \mathbb{R}$, by means of the elements of $J$,  listed in decreasing order from $m^+$ to $m^-$, so that matrix elements of $B$ read
\begin{equation*}
\label{matrix elements of B}
(B)_{pq}=(B_q)_p, \qquad p,q\in J.
\end{equation*}
It is easy to see that the adjoint $B^*\in \Psi^0_m$ of $B$ is given by
\begin{equation}
\label{B^*}
B^*: f \mapsto 
\begin{pmatrix}
B_{m^+}^*f
\\
\vdots
\\
B_1^*f
\\
B_{-1}^*f
\\
\vdots
\\
B_{-m^{-}}^{*}f
\end{pmatrix}, \qquad (B^*f)_p=B_p^*f.
\end{equation}

 Formulae \eqref{main theorem 1 equation 2}, \eqref{main theorem 1 equation 3} and Theorem~\ref{theorem results from part 1} immediately imply
\begin{equation}
\label{B*B}
B^*B=\operatorname{Id}_m \mod \Psi^{-\infty}_m.
\end{equation}
Proposition~\ref{proposition on Bj vs Pj} and Theorem~\ref{theorem results from part 1} imply
\begin{equation}
\label{BB^*}
BB^*=\operatorname{Id}_m \mod \Psi^{-\infty}_m.
\end{equation}
In view of \eqref{B*B} and \eqref{BB^*} we conclude that $B$ is almost-unitary, up to an infinitely smoothing operator.

Formulae \eqref{main theorem 1 equation 2} and \eqref{main theorem 1 equation 3}, combined with Theorem~\ref{theorem results from part 1}, yield
\begin{equation*}
\label{B*AB is diagonal}
\begin{split}
(B^*AB)_{lj}
&
=
B_l^*AB_j=B_l^*P_lAP_jB_j\mod \Psi^{-\infty}_1=B_l^*AP_lP_jB_j\mod\Psi^{-\infty}_1
\\
&
=
\delta_{lj} B_j^*AB_j \mod \Psi^{-\infty}_1,
\end{split}
\end{equation*}
namely, $\tilde A:=B^*AB$ is a diagonal operator modulo $\Psi^{-\infty}_m$. The operators appearing on the diagonal are elliptic self-adjoint pseudodifferential operators of order $s$ from $\mathcal{H}^s(M)$ to $\mathcal{L}^2(M)$ given by
\begin{equation}
\label{operators aj}
a_j:=B_j^*AB_j, \qquad (a_j)_\prin=h^{(j)}, \qquad j\in J.
\end{equation}

Hence,  we have obtained Theorem~\ref{main theorem 3}.

\begin{remark}
Observe that one can easily compute the subprincipal symbol of $\tilde A$ via \eqref{operators aj}, \eqref{main theorem 2 equation 1} and the identity
\begin{multline}
\label{subprincipal symbol of composition of three operators}
(PQR)_\sub=P_\sub Q_\prin R_\prin+P_\prin Q_\sub R_\prin+P_\prin Q_\prin R_\sub
\\
+\frac{i}2\left( \{P,Q\}R+\{P,Q,R\}+P\{Q,R\} \right).
\end{multline}
In fact, \eqref{operators aj} yields, on account of subsection~\ref{The algorithm},  an explicit algorithm for the computation of the full symbol of $\tilde A$.
\end{remark}

\subsection{Spectral analysis}
\label{Spectral analysis}

In order to prove Theorem~\ref{main theorem 4} and Theorem~\ref{main theorem 5}, we will rely on a strategy developed by the author and Vassiliev in \cite{part2}.  Throughout this section, all estimates are to be understood as asymptotic estimates as $k\to+\infty$, unless otherwise stated.

Let us introduce the auxiliary operators
\begin{equation}
\label{definition of Aj}
A_j:=P_j^*AP_j-\operatorname{sgn}(j)\sum_{l\ne j}\operatorname{sgn}(l) P_l^* AP_l, \qquad j\in J.
\end{equation}
Here and further on $\operatorname{sgn}$ returns the sign the nonzero integer number it acts upon. The operators $A_j\in \Psi^{s}_m$ are elliptic and coincide,  for $j>0$ and modulo $\Psi^{-\infty}_m$, with the operators $A_j$ introduced in \cite[Sec.~2]{part2}.  Observe that $(A_j)_\prin$ has precisely one positive (resp.~negative) eigenvalue for positive (resp.~negative) $j$ --- the eigenvalue $h^{(j)}$.

Conjugating $A_j$ by $B$ we get
\begin{equation*}
\label{conjugating Aj by U matrix elements}
(B^*A_jB)_{lr}
=
\delta_{lr} \delta_{lj} \,a_j -\operatorname{sgn}(j\cdot l) \delta_{lr} (1-\delta_{lj})\,a_l \mod \Psi^{-\infty}_1,
\end{equation*}
\begin{equation}
\label{conjugating Aj by U full matrix}
B^*A_jB=\begin{cases}
\begin{pmatrix}
-a_{m^+} & & & &\\
& a_j & & &\\
& & -a_{1}& && \\
& & &a_{-1}& & \\
& & & & \ddots &\\
& & & & & a_{m^-}
\end{pmatrix} \mod \Psi^{-\infty}_m&\ j>0,\\
\\
\begin{pmatrix}
a_{m^+} & & & &\\
& \ddots& & &\\
& & a_{1}& && \\
& & & -a_{-1}& & \\
& & & & a_j &\\
& & & & & -a_{m^-}
\end{pmatrix} \mod \Psi^{-\infty}_m &\ j<0.
\end{cases}
\end{equation}

Suppose $j>0$. Then \eqref{conjugating Aj by U full matrix} suggests that the positive spectrum of $A_j$ is asymptotically described by the spectrum of $a_j$ and {\it vice versa}. In the remainder of this section we will provide a rigorous proof for this,  accounting for the fact that all the above identities are only modulo smoothing operators. The latter will result in the appearance of remainder terms.

\

Our strategy goes as follows.  First, we relate the spectra of $a_j$, $j\in J$, with the spectra of $A_j$, $j\in J$.  Then, we exploit the results on the relation between the spectra of $A_j$, $j\in J$, and the spectrum of $A$ obtained in \cite{part2} to establish a relation between the spectra of $a_j$, $j\in J$, and the spectrum of $A$.

\begin{remark}
We should like to point out that one can prove Theorems~\ref{main theorem 4} and~\ref{main theorem 5} directly, without the need to introduce the operators $A_j$ and rely on results from \cite{part2}. As this would make the proof substantially longer without bringing about an equivalent benefit in terms of insight, we decided against going down this path. Moreover, the operators $A_j$ are, generally speaking, more suitable to study refined spectral properties of $A$ --- see subsection~\ref{The second Weyl coefficient}; relating the spectra of $A_j$ and $a_j$ contributes some additional information which may be useful in future works.
\end{remark}

Without loss of generality, we will prove our results for $j>0$. The case $j<0$ is similar,  and the corresponding proof follows from the case $j>0$ by replacing $A$ with $-A$.

Let
\begin{equation}
\label{positive spectrum Aj}
0<\lambda^{(j)}_1 \le \lambda^{(j)}_2 \le \ldots\le \lambda^{(j)}_k\le \ldots \to +\infty
\end{equation}
be the positive eigenvalues of $A_j$, enumerated with account of multiplicity.

\begin{notation}
When we write
\[
f=O(k^{-\infty}),
\]
the estimate is to be understood as follows: any given partial derivative of the quantity on the LHS is estimated by any given negative power of $k$ uniformly over $M$.
\end{notation}

\begin{theorem}
\label{theorem about closeness of spectra aj vs Aj}
For $j\in J$, $j>0$, we have
\begin{enumerate}[(i)]
\item 
$\operatorname{dist}(\ell_k^{(j)},\sigma^+(A_j))=O(k^{-\infty})\quad\text{as}\quad k\to+\infty$,

\item
$\operatorname{dist}(\lambda_k^{(j)},\sigma^+(a_j))=O(k^{-\infty})\quad\text{as}\quad k\to+\infty$,
\end{enumerate}
where $\ell_k^{(j)}$ and $\lambda_k^{(j)}$ are defined in accordance with \eqref{positive spectrum a_j} and \eqref{positive spectrum Aj}, respectively.
\end{theorem}

\begin{proof}
(i) Let $f^{(j)}_k$ be a normalised eigenfunction of $a_j$ corresponding to the eigenvalue $\ell^{(j)}_k$,
\begin{equation}
\label{proof theorem closeness of spectra j version equation 1}
a_j f^{(j)}_k=\ell^{(j)}_k f^{(j)}_k, \qquad \|f^{(j)}_k\|_{\mathcal{L}^2}=1.
\end{equation}
For every $r=1,2,\ldots$ and every $Q\in \Psi^{-\infty}_{1,r}$,  we have
\begin{equation}
\label{proof theorem closeness of spectra j version equation 2}
Q f^{(j)}_k=O(k^{-\infty}).
\end{equation}
Formula \eqref{proof theorem closeness of spectra j version equation 2} follows from Weyl's Law \cite[Theorem~B.1]{part2}
\begin{equation*}
\label{proof theorem closeness of spectra j version equation 3}
\ell^{(j)}_k= c \,k^{s/d}+ o(k^{s/d})
\end{equation*}
(recall that $A$ is an operator of order $s$) and the identity
\begin{equation*}
\label{proof theorem closeness of spectra j version equation 4}
Q f^{(j)}_k=(\ell^{(j)}_k)^{-n} Q (a_j)^n f^{(j)}_k \qquad \forall n\in \mathbb{N}.
\end{equation*}
Put $v^{(j)}_k:=B_j f^{(j)}_k$. We claim that
\begin{equation}
\label{proof theorem closeness of spectra j version equation 5}
A_j v^{(j)}_k=\ell^{(j)}_k v^{(j)}_k+O(k^{-\infty}).
\end{equation}
Indeed,  taking into account \eqref{definition of Aj}, \eqref{proof theorem closeness of spectra j version equation 2} and \eqref{main theorem 1 equation 3}, for $l\in J\setminus\{j\}$ we have
\begin{equation*}
\label{proof theorem closeness of spectra j version equation 6}
\begin{split}
B_l^*A_j v^{(j)}_k
&
=
B_l^*P_l^*A_j P_j v^{(j)}_k+O(k^{-\infty})
=
O(k^{-\infty}).
\end{split}
\end{equation*}
Hence,  \eqref{proposition on Bj vs Pj equation 2} implies
\begin{equation*}
\label{proof theorem closeness of spectra j version equation 7}
A_j v^{(j)}_k=B_jB_j^*A_j v^{(j)}_k+O(k^{-\infty}),
\end{equation*}
which, combined with \eqref{proof theorem closeness of spectra j version equation 1}, \eqref{proof theorem closeness of spectra j version equation 2} and the definition of $v^{(j)}_k$,  yields \eqref{proof theorem closeness of spectra j version equation 5}.

Formulae \eqref{proof theorem closeness of spectra j version equation 5}, \eqref{proof theorem closeness of spectra j version equation 1} and the estimate
\begin{equation*}
\label{proof theorem closeness of spectra j version equation 8}
\|v^{(j)}_k\|_{L^2}^2=\langle B_j f^{(j)}_k, B_j f^{(j)}_k \rangle=\prec\! B_j^* B_j f^{(j)}_k, f^{(j)}_k \!\succ =\|f^{(j)}_k\|_{\mathcal{L}^2}^2+O(k^{-\infty})
\end{equation*}
(recall \eqref{main theorem 1 equation 2}) give us (i).

\

(ii) Let $v^{(j)}_k$ be a normalised eigenfunction of $A_j$ corresponding to the eigenvalue $\lambda^{(j)}_k$,
\begin{equation}
\label{proof theorem closeness of spectra j version equation 9}
A_j v^{(j)}_k=\lambda^{(j)}_k v^{(j)}_k, \qquad \|v^{(j)}_k\|_{L^2}=1.
\end{equation}
Arguing as in part (i), one obtains that, for every $r=1,2,\ldots$ and every $Q\in \Psi^{-\infty}_{m,r}$,
\begin{equation}
\label{proof theorem closeness of spectra j version equation 11}
Qv^{(j)}_k=O(k^{-\infty}).
\end{equation}
It was shown in \cite[Proposition~3.1]{part2} that
\begin{equation*}
\label{proof theorem closeness of spectra j version equation 12}
P_lv^{(j)}_k=O(k^{-\infty})\quad\text{for}\quad l\in J\setminus\{j\}.
\end{equation*}
The latter, alongside Proposition~\ref{proposition on Bj vs Pj} and \eqref{proof theorem closeness of spectra j version equation 11}, implies
\begin{equation*}
\label{proof theorem closeness of spectra j version equation 13}
B_lB_l^* v^{(j)}_k=O(k^{-\infty})\quad\text{for}\quad l\in J\setminus\{j\}
\end{equation*}
and
\begin{equation}
\label{proof theorem closeness of spectra j version equation 14}
v^{(j)}_k=B_j B_j^* v^{(j)}_k+O(k^{-\infty}).
\end{equation}
Combining \eqref{proof theorem closeness of spectra j version equation 9} and \eqref{proof theorem closeness of spectra j version equation 14}, and using once again \eqref{proof theorem closeness of spectra j version equation 11}, one obtains
\begin{equation*}
\label{proof theorem closeness of spectra j version equation 15}
A_jB_j B_j^* v^{(j)}_k=\lambda^{(j)}_k B_j B_j^* v^{(j)}_k+O(k^{-\infty}).
\end{equation*}
Composing the above equation with $B_j^*$ on the left, and using \eqref{main theorem 1 equation 2} and \eqref{proof theorem closeness of spectra j version equation 11}, we arrive at
\begin{equation}
\label{proof theorem closeness of spectra j version equation 16}
a_j B_j^* v^{(j)}_k=\lambda^{(j)}_k  B_j^* v^{(j)}_k+O(k^{-\infty}).
\end{equation}
As, on account of \eqref{proof theorem closeness of spectra j version equation 14}, we have
\begin{equation*}
\label{proof theorem closeness of spectra j version equation 17}
\|B_j^* v^{(j)}_k\|_{\mathcal{L}^2}=\|v^{(j)}_k\|_{L^2}+O(k^{-\infty}),
\end{equation*}
formulae \eqref{proof theorem closeness of spectra j version equation 9} and \eqref{proof theorem closeness of spectra j version equation 16} give us (ii).
\end{proof}

Theorem~\ref{theorem about closeness of spectra aj vs Aj} establishes closeness of the spectra of $a_j$ and $A_j$. The following theorem goes a step further and establishes closeness of individual eigenvalues.

\begin{theorem}
\label{theorem about closeness of eigenvalues aj vs Aj}
For every $\alpha>0$ there exists $z_\alpha\in \mathbb{Z}$ such that
\begin{equation}
\label{theorem about closeness of eigenvalues aj vs Aj equation 1}
\ell^{(j)}_k=\lambda^{(j)}_{k+z_\alpha}+O(k^{-\alpha}) \qquad \forall j=1, \ldots, m^+.
\end{equation}
\end{theorem}

\begin{proof}
Let us partition the positive real line $(0,+\infty)$ into intervals $(\nu_n, \nu_{n+1}]$, $n=0,1,2,\ldots$,  whose endpoints are defined in accordance with
\begin{equation*}
\label{proof theorem about closeness of eigenvalues aj vs Aj equation 1}
\nu_0=0, \qquad \nu_n:=n^\beta+c_n \,n^{-1}, \quad n=1,2,\ldots,
\end{equation*}
where
\begin{equation*}
\label{proof theorem about closeness of eigenvalues aj vs Aj equation 2}
\beta:=\frac{1}{1+\frac{\alpha d}{s}}
\end{equation*}
and the constants $c_n\in [-\beta/4, \beta/4]$ are chosen in such a way that
\begin{equation}
\label{proof theorem about closeness of eigenvalues aj vs Aj equation 3}
\min_{j=1,\ldots, m^+} \min\{\operatorname{dist}(\nu_n,\sigma^+(a_j)),\operatorname{dist}(\nu_n,\sigma^+(A_j))\} \ge C n^{-\gamma}
\end{equation}
for some constant $C$ and 
\begin{equation*}
\label{proof theorem about closeness of eigenvalues aj vs Aj equation 4}
\gamma:=1+\frac{d \beta}{s}.
\end{equation*}
That such a partition exists follows in a straightforward manner from the results in \cite[Subsec.~3.2.1]{part2}.

A simple asymptotic estimate tells us that
\begin{equation}
\label{estimate on size of partition}
\nu_{n+1}-\nu_n=O(\nu_n^{-\frac{\alpha d}s})\quad \text{as}\quad n\to+\infty.
\end{equation}
Our strategy consists in proving that, asymptotically, each interval $(\nu_n,\nu_{n+1}]$ contains the same number of eigenvalues $\ell^{(j)}_j$ and $\lambda^{(j)}_k$, and then using \eqref{estimate on size of partition} to achieve \eqref{theorem about closeness of eigenvalues aj vs Aj equation 1} with the desired accuracy --- see Figure~\ref{figure partition}. To this end, let us define
\begin{equation*}
\label{number of ell in each interval}
N_j(n):=\#\{\lambda^{(j)}_k\ :\ \nu_n\le \lambda^{(j)}_k \le \nu_{n+1}\},
\end{equation*}
\begin{equation*}
\label{number of lambdas in each interval}
n_j(n):=\#\{\ell^{(j)}_k\ :\ \nu_n\le \ell^{(j)}_k \le \nu_{n+1}\}.
\end{equation*}


\begin{figure}
\centering
\begin{tikzpicture}[scale=1.15]

\draw[-{>[scale=3.5,
          length=2,
          width=2]}
] (0,0)--(13,0);

\draw [black,thick,domain=160:200] plot ({1*cos(\x)+2.5}, {1*sin(\x)});

\draw [thick] (10.4,-.33)--(10.5,-.33)--(10.5,.33)--(10.4,.33);

\node [above] at (1.5,.33) {$\nu_{n}=n^\beta+c_n\,n^{-1}$};
\node [above] at (10.5,.33) {$\nu_{n+1}$};

\draw[fill,red] (3,0) circle [radius=0.05];
\draw[fill,red] (6.5,0) circle [radius=0.05];
\draw[fill,red] (7.8,0) circle [radius=0.05];

\node [below,red] at (3,-.1) {$\lambda^{(j)}_{k+1}$};
\node [below,red] at (7.8,-.1) {$\lambda^{(j)}_{k+N_j(n)}$};

\draw[fill,blue] (.25,0) circle [radius=0.05];
\draw[fill,blue] (4,0) circle [radius=0.05];
\draw[fill,blue] (5,0) circle [radius=0.05];
\draw[fill,blue] (9.5,0) circle [radius=0.05];
\draw[fill,blue] (12.2,0) circle [radius=0.05];

\node [below,blue] at (4,-.1) {$\ell^{(j)}_{k+r_\alpha+1}$} ;
\node [below,blue] at (9.5,-.1) {$\ell^{(j)}_{k+r_\alpha+n_j(n)}$} ;

\draw [thick, decorate,decoration={brace,amplitude=6pt,mirror}] (.25,-.9) -- (1.45,-.9);
\node [below] at (.87,-1.2) {$\gtrsim n^{-\gamma}$};
\draw [thick, decorate,decoration={brace,amplitude=6pt,mirror}] (1.5,-.9) -- (3,-.9);
\node [below] at (2.25,-1.2) {$\gtrsim n^{-\gamma}$};

\draw [thick, decorate,decoration={brace,amplitude=6pt,mirror}] (9.5,-.9) -- (10.5,-.9);
\node [below] at (10,-1.2) {$\gtrsim n^{-\gamma}$};
\draw [thick, decorate,decoration={brace,amplitude=6pt,mirror}] (10.55,-.9) -- (12.2,-.9);
\node [below] at (11.4,-1.2) {$\gtrsim n^{-\gamma}$};

\draw [thick, decorate,decoration={brace,amplitude=9pt}] (1.5,1.2) -- (10.5,1.2);
\node [above] at (6,1.8) {$\sim \nu_n^{-\frac{\alpha d}{s}}$};


\end{tikzpicture}
\caption{Partition of the positive real line
}
\label{figure partition}
\end{figure}
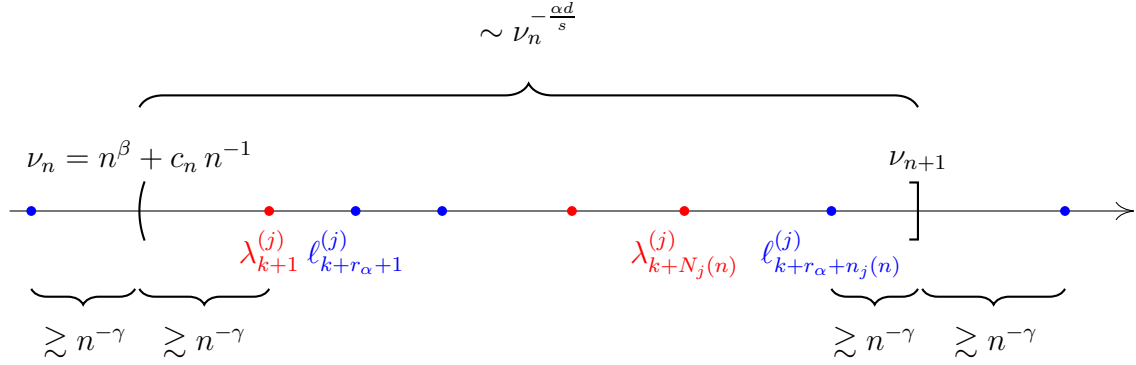

Let us begin by showing that for sufficiently large $n$ we have
\begin{equation}
\label{proof theorem about closeness of eigenvalues aj vs Aj equation 5}
N_j(n)\le n_j(n).
\end{equation}
Suppose $N_j(n)=r$.  Let $\lambda^{(j)}_{q+1},\le\ldots,\le \lambda^{(j)}_{q+r}$ be the $r$ eigenvalues of $A_j$ in $(\nu_n, \nu_{n+1}]$ (recall that there is a neighbourhood of the endpoints free from eigenvalues) and let $v_1, \ldots, v_r$ be the corresponding orthonormalised eigenfunctions,
\[
\langle v_k, v_{k'} \rangle=\delta_{kk'}, \qquad k,k'=1, \ldots, r.
\]
It follows from \eqref{proof theorem closeness of spectra j version equation 14} that
\begin{equation}
\label{proof theorem about closeness of eigenvalues aj vs Aj equation 6}
\prec\! B_j^* v_k, B_j^* v_{k'} \! \succ= \delta_{kk'} +O(n^{-\infty})\quad\text{as}\quad n\to+\infty.
\end{equation}
Let $F$ be the $r\times r$ matrix whose matrix elements are given by
\begin{equation}
\label{proof theorem about closeness of eigenvalues aj vs Aj equation 7}
F_{kk'}:=\prec\! B_j^* v_k, B_j^* v_{k'} \! \succ, \qquad k,k'=1,\ldots, r,
\end{equation}
and let $G$ be the matrix given by Lemma~\ref{lemma about orthonormalisation}.  Define
\begin{equation*}
f_k:=\sum_{l=1}^r G_{lk}\, B_j^*v_l.
\end{equation*}
Then formula \eqref{lemma about orthonormalisation equation 2} implies
\begin{equation}
\label{proof theorem about closeness of eigenvalues aj vs Aj equation 8}
\prec\! f_k, f_{k'}   \! \succ=\delta_{kk'}, \qquad k,k'=1,\ldots, r,
\end{equation}
whereas formulae \eqref{lemma about orthonormalisation equation 3}, \eqref{proof theorem about closeness of eigenvalues aj vs Aj equation 6} and elliptic regularity theory imply
\begin{equation}
\label{proof theorem about closeness of eigenvalues aj vs Aj equation 9}
f_k=B_j^* v_k+O(n^{-\infty})\quad\text{as}\quad n\to+\infty.
\end{equation}
Observe that one cannot achieve \eqref{proof theorem about closeness of eigenvalues aj vs Aj equation 8}--\eqref{proof theorem about closeness of eigenvalues aj vs Aj equation 9} via the Gram--Schmidt process,  because the latter involves a number of terms growing factorially with $r$,  which would produce too big an error in \eqref{proof theorem about closeness of eigenvalues aj vs Aj equation 9}.

Combining \eqref{proof theorem about closeness of eigenvalues aj vs Aj equation 9} and \eqref{proof theorem closeness of spectra j version equation 16} we conclude that for every $\tilde{\gamma}>0$ there exists a constant $C_{\tilde{\gamma}}$ such that
\begin{equation}
\label{proof theorem about closeness of eigenvalues aj vs Aj equation 10}
\|(a_j-\lambda^{(j)}_{q+k}) f_k\|_{\mathcal{L}^2}\le C_{\tilde{\gamma}}\, n^{-\tilde \gamma}, \qquad k=1,\ldots, r.
\end{equation}

In view of \eqref{proof theorem about closeness of eigenvalues aj vs Aj equation 10} and \eqref{proof theorem about closeness of eigenvalues aj vs Aj equation 8} we can apply Lemma~\ref{lemma about bound on number of ev in interval} to get
\begin{equation}
\label{proof theorem about closeness of eigenvalues aj vs Aj equation 11}
\#\{\ell^{(j)}_k\ |\ \ell^{(j)}_k \in [\lambda^{(j)}_{q+1}-C_{\tilde{\gamma}}\sqrt{r}n^{-\tilde{\gamma}},\lambda^{(j)}_{q+r}+C_{\tilde{\gamma}}\sqrt{r}n^{-\tilde{\gamma}}]\}\ge r\,.
\end{equation}
Now, Weyl's law tells us that $r=O(n^{\beta d/s})$. Therefore,  if $\tilde{\gamma}>\tfrac{\beta d}{2s}+\gamma$,  condition \eqref{proof theorem about closeness of eigenvalues aj vs Aj equation 3} implies that for sufficiently large $n$ we have
\begin{equation}
\label{proof theorem about closeness of eigenvalues aj vs Aj equation 12}
[\lambda^{(j)}_{q+1}-C_{\tilde{\gamma}}\sqrt{r}n^{-\tilde{\gamma}},\lambda^{(j)}_{q+r}+C_{\tilde{\gamma}}\sqrt{r}n^{-\tilde{\gamma}}] \subset (\nu_n, \nu_{n+1}).
\end{equation}
Formulae \eqref{proof theorem about closeness of eigenvalues aj vs Aj equation 11} and \eqref{proof theorem about closeness of eigenvalues aj vs Aj equation 12} give us \eqref{proof theorem about closeness of eigenvalues aj vs Aj equation 5}.

Next, let us show that for sufficiently large $n$ we have
\begin{equation}
\label{proof theorem about closeness of eigenvalues aj vs Aj equation 13}
n_j(n)\le N_j(n).
\end{equation}
Suppose $n_j(n)=r$.  Let $\ell^{(j)}_{q+1},\le\ldots,\le \ell^{(j)}_{q+r}$ be the $r$ eigenvalues of $a_j$ in $(\nu_n, \nu_{n+1}]$ (recall that there is a neighbourhood of the endpoints free from eigenvalues) and let $f_1, \ldots, f_r$ be the corresponding orthonormalised eigenfunctions,
\begin{equation}
\label{proof theorem about closeness of eigenvalues aj vs Aj equation 14}
\langle f_k, f_{k'} \rangle=\delta_{kk'}, \qquad k,k'=1, \ldots, r.
\end{equation}
Conditions \eqref{main theorem 1 equation 2}--\eqref{main theorem 1 equation 3}, Theorem~\ref{theorem results from part 1} and \eqref{proof theorem about closeness of eigenvalues aj vs Aj equation 14} imply 
\begin{equation*}
\label{proof theorem about closeness of eigenvalues aj vs Aj equation 15}
\langle B_j f_k, B_j f_{k'} \rangle=\delta_{kk'}+O(n^{-\infty}) \quad\text{as}\quad n\to+\infty.
\end{equation*}
Let $F$ be the matrix whose matrix elements are given by
\begin{equation*}
\label{proof theorem about closeness of eigenvalues aj vs Aj equation 16}
F_{kk'}:=\langle B_j f_k, B_j f_{k'} \rangle, \qquad k,k'=1,\ldots,r,
\end{equation*}
and let $G$ be the matrix given by Lemma~\ref{lemma about orthonormalisation}.  Arguing as above and using \eqref{proof theorem closeness of spectra j version equation 5}, we conclude that the $m$-columns
\begin{equation*}
\label{proof theorem about closeness of eigenvalues aj vs Aj equation 17}
v_k:=\sum_{l=1}^r G_{lk}\, B_j f_l, \qquad k=1,\ldots, r,
\end{equation*}
satisfy
\begin{equation*}
\label{proof theorem about closeness of eigenvalues aj vs Aj equation 18}
\langle v_k, v_{k'} \rangle=\delta_{kk'}, \qquad k,k'=1,\ldots,r, 
\end{equation*}
and for every $\tilde{\gamma}>0$ there exists $C_{\tilde \gamma}>0$ such that
\begin{equation*}
\label{proof theorem about closeness of eigenvalues aj vs Aj equation 19}
\|(A_j-\ell^{(j)}_{q+k})v_k\|_{L^2}\le C_{\tilde \gamma} \,n^{-\tilde{\gamma}}.
\end{equation*}
Using once again Lemma~\ref{lemma about bound on number of ev in interval} one obtains analogues of \eqref{proof theorem about closeness of eigenvalues aj vs Aj equation 11} and \eqref{proof theorem about closeness of eigenvalues aj vs Aj equation 12},  which in turn imply \eqref{proof theorem about closeness of eigenvalues aj vs Aj equation 13}.

All in all,  on account of 
\eqref{proof theorem about closeness of eigenvalues aj vs Aj equation 5}
and
\eqref{proof theorem about closeness of eigenvalues aj vs Aj equation 13},
we have established that there exists $K\in \mathbb{N}$ such that for all $n>K$ we have
\begin{equation}
\label{proof theorem about closeness of eigenvalues aj vs Aj equation 20}
N_j(n)=n_j(n).
\end{equation}
Suppose $\lambda^{(j)}_k\in(\nu_n, \nu_{n+1}]$.  Then the Weyl's Law gives us
\begin{equation}
\label{proof theorem about closeness of eigenvalues aj vs Aj equation 21}
\nu_n=c\, k^{\frac{s}{d}}+o(k^{\frac{s}{d}}),
\end{equation}
where $c$ is a constant that can be explicitly computed. We can use \eqref{proof theorem about closeness of eigenvalues aj vs Aj equation 21} to equivalently recast \eqref{estimate on size of partition} as
\begin{equation}
\label{proof theorem about closeness of eigenvalues aj vs Aj equation 22}
\nu_{n+1}-\nu_n=O(k^{-\alpha}).
\end{equation}
Combining formulae \eqref{proof theorem about closeness of eigenvalues aj vs Aj equation 20} and \eqref{proof theorem about closeness of eigenvalues aj vs Aj equation 22} we arrive at \eqref{theorem about closeness of eigenvalues aj vs Aj equation 1}
\end{proof}

\begin{remark}
Observe that Theorem~\ref{theorem about closeness of spectra aj vs Aj} follows from Theorem~\ref{theorem about closeness of eigenvalues aj vs Aj}; we stated the two theorems as separate results for the sake of logical clarity.  Also note that the additional material is not completely unnecessary: intermediate results obtained in the proof of the former theorem are needed in the proof of the latter.
\end{remark}

We are now in a position to prove Theorems~\ref{main theorem 4} and~\ref{main theorem 5}.

\begin{proof}[Proof of Theorem~\ref{main theorem 4}]
On account of \eqref{definition of Aj} and \cite[Eqn.~(2.2)]{part2},  Theorem~\ref{theorem about closeness of spectra aj vs Aj}(ii) and \cite[Theorem~2.2]{part2} imply \eqref{main theorem 4 equation 1}, whereas Theorem~\ref{theorem about closeness of spectra aj vs Aj}(i) and \cite[Theorem~2.1]{part2} imply \eqref{main theorem 4 equation 2}.
\end{proof}

\begin{proof}[Proof of Theorem~\ref{main theorem 5}]
On account of \eqref{definition of Aj} and \cite[Eqn.~(2.2)]{part2},  formula \eqref{main theorem 5 equation 1} follows from Theorem~\ref{theorem about closeness of eigenvalues aj vs Aj} and \cite[Theorem~2.3]{part2}.
\end{proof}

\subsection{Spectral asymptotics}
\label{The second Weyl coefficient}

In this section we will discuss the relation between the Weyl coefficients of $A$ and those of its diagonalised version $\tilde A$ in the special case of a first order operator. Throughout this section we set $s=1$ and, without loss of generality, we assume $m^+\ge 1$.

Let $v_k$ be a sequence of orthonormalised eigenfunctions of the operator $A$ corresponding to the sequence of eigenvalues \eqref{positive spectrum A}.  Let $\tilde v_k$ be a sequence of orthonormalised eigenfunctions of the operator $\tilde A$ corresponding to the sequence of eigenvalues \eqref{spectra of a_j combined}.

Let us define the spectral densities
\begin{equation}
\label{spectral density A}
N(x,\lambda):=
\begin{cases}
0 & \text{if}\ \lambda\le 0,\\
\displaystyle\sum_{k\ :\ 0<\lambda_k<\lambda} [v_k(x)]^* v_k(x) & \text{if}\ \lambda>0,
\end{cases}
\end{equation}
and
\begin{equation}
\label{spectral density tilde A}
\tilde N(x,\lambda):=
\begin{cases}
0 & \text{if}\ \lambda\le 0,\\
\displaystyle \sum_{k\ :\ 0<\zeta_k<\lambda} [\tilde v_k(x)]^*\tilde v_k(x) & \text{if}\ \lambda>0.
\end{cases}
\end{equation}
The quantities \eqref{spectral density A} and \eqref{spectral density tilde A} are sometimes referred to as \emph{local counting functions}, because they count the number of positive eigenvalues of the operator below a given $\lambda$, with  $x$-dependent (hence `local') weight given by the modulus squared of the corresponding eigenfunctions. 

Integrating \eqref{spectral density A} and \eqref{spectral density tilde A} one obtains the usual (global positive) counting functions 
\begin{equation}
\label{counting function A}
N(\lambda):=\int_M N(x,\lambda)\,dx=
\begin{cases}
0 & \text{if}\ \lambda\le 0,\\
\displaystyle \sum_{k\ :\ 0<\lambda_k<\lambda} 1 & \text{if}\ \lambda>0,
\end{cases}
\end{equation}
and
\begin{equation}
\label{counting function tilde A}
\tilde N(\lambda):=\int_M \tilde N(x,\lambda)\,dx=
\begin{cases}
0 & \text{if}\ \lambda\le 0,\\
\displaystyle \sum_{k\ :\ 0<\zeta_k<\lambda}1 & \text{if}\ \lambda>0.
\end{cases}
\end{equation}
Note that \eqref{spectral density A} and \eqref{spectral density tilde A} are densities over $M$, whereas \eqref{counting function A} and \eqref{counting function tilde A} are scalar functions.

For each $j\in J$ let 
\begin{equation}
\label{hamiltonian flow}
t \mapsto (x^{(j)}(t;y,\eta), \xi^{(j)}(t;y,\eta))
\end{equation}
be the solution to Hamilton's equations
\begin{equation*}
\dot x^{(j)}=h^{(j)}_\xi(x^{(j)},\xi^{(j)}), \qquad \dot \xi^{(j)}=-h^{(j)}_x(x^{(j)},\xi^{(j)})
\end{equation*}
with initial condition $(x^{(j)}(0;y,\eta), \xi^{(j)}(0;y,\eta))=(y,\eta)\in T'M$.
Define $T_0$ to be the infimum of lengths of all the Hamiltonian loops generated by the Hamiltonian flow \eqref{hamiltonian flow} for all $j\in J$ and all possible initial conditions.
Let $\hat \mu : \mathbb{R} \to \mathbb{C}$ be a smooth function such that $\mu \equiv 1$ in a neighbourhood of $0$ and $\operatorname{supp} \hat \mu \subset (-T_0, T_0)$. Let $\mu:=\mathcal{F}^{-1}_{t\to \lambda}\hat \mu$ be the inverse Fourier transform of $\hat \mu$. 

It is well known \cite{DuGu,Ivr80,Ivr84,Ivr98, rozenblum,safarov,SaVa} that the mollified derivative of the spectral density admits a 
complete asymptotic expansion in integers powers of $\lambda$:
\begin{equation}
\label{mollified derivative A}
(N'*\mu)(x,\lambda)=w_{d-1}(x)\lambda^{d-1}+w_{d-2}(x)\lambda^{d-2}+ \ldots \quad \text{as}\quad \lambda\to +\infty,
\end{equation}
\begin{equation}
\label{mollified derivative tilde A}
(\tilde N'*\mu)(x,\lambda)=\tilde w_{d-1}(x)\lambda^{d-1}+\tilde w_{d-2}(x)\lambda^{d-2}+ \ldots \quad \text{as}\quad \lambda\to +\infty.
\end{equation}
Here the prime denotes the derivative in $\lambda$ and $*$ convolution with respect to $\lambda$. The smooth densities appearing as coefficients of powers of $\lambda$ in \eqref{mollified derivative A} and \eqref{mollified derivative A} are usually called \emph{local Weyl coefficients}. Integrating the local Weyl coefficients over $M$ one obtains scalars known as (global) \emph{Weyl coefficients}.

Computing Weyl coefficients is a delicate task.  A formula for the second local Weyl coefficient $w_{d-2}$ for systems ($m\ge 2$) was first obtained in 2013 by Chervova, Downes and Vassiliev \cite{CDV}, fixing decades of incorrect or incomplete publications in the subject (see \cite[Section~11]{CDV}), and later recovered by a completely different method by Avetisyan, Sj\"ostrand and Vassiliev \cite{ASV}.  The formula for $w_{d-2}$ reads \cite[Eqn~(1.24)]{CDV} \cite[Eqn.~(1.6)]{ASV} \cite[Eqn.~(5.10)]{part2}
\begin{multline}
\label{second Weyl coefficient for A}
w_{d-2}^{(j)}(x)=-\frac{d(d-1)}{(2\pi)^d}
\ \int\limits_{h^{(j)}(x,\xi)<1}
\operatorname{tr}\left(
P^{(j)}A_\mathrm{sub}
+\frac i2\{P^{(j)}, P^{(j)}\}A_\mathrm{prin}
\right.
\\
\left.
-\frac {1}{d-1}h^{(j)}(P_j)_\mathrm{sub}
\right)
(x,\xi)\,
d\xi\,.
\end{multline}
Working with systems, as opposed to scalar equations, introduces substantial difficulties, which one needs to overcome and which are simply not present when studying the spectrum of scalar problems. 

Now, the diagonalization argument presented above effectively reduces our original system $A$ to a collection of scalar operators, the diagonal operator $\tilde{A}$.  It is then natural to ask whether the approximate diagonalization offers a way to provide a different proof for \eqref{second Weyl coefficient for A}, one that relies on the corresponding formula for scalar operators. Indeed, it is easy to see that the unique (modulo $\Psi^{-\infty}_m$) pseudodifferential basis given by Theorem~\ref{theorem results from part 1} associated with the operator $\tilde{A}$ is
\begin{equation}
\label{Pj tilde}
\tilde{P}_j:=B^*P_jB
=\tilde{v}^{(j)}[\tilde{v}^{(j)}]^*\operatorname{Id} \mod \Psi^{-\infty}, \qquad j\in J,
\end{equation}
where the $\tilde v^{(j)}$,
\begin{equation}
\label{eigenvectors tilde A prin}
[\tilde v^{(j)}]_q=\delta_{jq}, \qquad j,q\in J,
\end{equation}
are the eigenvectors of $\tilde A_\prin$. Formula \eqref{Pj tilde} immediately implies
\begin{equation}
\label{subprincipal tilde Pj}
(\tilde P_j)_\sub=0, \qquad j\in J.
\end{equation}
On account of \eqref{eigenvectors tilde A prin} and \eqref{subprincipal tilde Pj}, formula \eqref{second Weyl coefficient for A} applied to the operator $\tilde A$ gives us
\begin{equation*}
\label{second weyl coefficient for A tilde}
\tilde w_{d-2}^{(j)}(x)=-\frac{d(d-1)}{(2\pi)^d}\sum_{j=1}^{m^+}
\ \int\limits_{h^{(j)}(x,\xi)<1}
(a_j)_\sub
(x,\xi)\,
d\xi\,,
\end{equation*}
namely,  as expected, the second Weyl coefficient for $\tilde A$ is the sum of the second Weyl coefficients of the scalar operators $a_j$, $j=1,\ldots, m^+$,  computed in accordance with the classical Duistermaat and Guillemin's formula for scalar elliptic operators \cite[Eqn.~(2.2)]{DuGu}.

Unfortunately, this doesn't quite work: the operators $A$ and $\tilde{A}$ are related via conjugation by the pseudodifferential operator $B$. Even though, as we have shown, this does not change the spectrum of $A$ too much --- where `not too much' is quantified by the remainder term in \eqref{main theorem 5 equation 1} ---
there is no reason why it should preserve the asymptotic expansion of the spectral density \eqref{spectral density A}. 
The diagonalization argument as presented in this paper only allows one to recover {\it global} spectral asymptotics.
In other words,  one can exploit diagonalization to recover in a simpler way --- i.e.~relying only on formulae for scalar operators --- the integral over $M$ of \eqref{second Weyl coefficient for A}, but not \eqref{second Weyl coefficient for A} itself.


\

To shed additional light onto the above discussion,  let us compare the approach pursued in this paper with that of \cite{part2}.  In \cite{part2} the author and Vassiliev showed that the spectrum of an elliptic system $A$ partitions into $m$ series of eigenvalues,  labelled by the eigenvalues $h^{(j)}$ of $A_\prin$, up to a superpolynomial error.  This mirrors, at the spectral level, the decomposition of $L^2(M)$ into almost-orthogonal almost-invariant subspaces via the pseudodifferential projections $P_j$, see~Theorem~\ref{theorem results from part 1}. This result was achieved by showing that the positive spectrum of $A$ is approximated by the union of the positive spectra of (a minor modification of) the elliptic matrix operators $A_j$, $j=1, \ldots,m^+$, defined in accordance with~\eqref{definition of Aj}.  For a first order system, it was also shown that the positive part of the propagator $U(t):=e^{-itA}$ decomposes, in a similar manner, as
\begin{equation}
\label{propagator relation 1}
\theta(A)\,U(t)=\sum_{j=1}^{m^+} U^{(j)}(t) \mod C^{\infty}(\mathbb{R}; \Psi^{-\infty}_m),
\end{equation}
where
\begin{equation}
\label{propagator relation 2}
U^{(j)}(t)=P_j U(t)=U(t)P_j= \theta(A_j)\,e^{-itA_j}\mod C^{\infty}(\mathbb{R}; \Psi^{-\infty}_m), \qquad j=1,\ldots, m^+.
\end{equation}
Here $\theta(\cdot)$ is the Heaviside theta function, see \cite[Sec.~7]{part1}.
We refer the reader to \cite{CDV,wave, dirac,lorentzian} for additional details on $U(t)$, $U^{(j)}(t)$ and their explicit construction. Properties \eqref{propagator relation 1} and \eqref{propagator relation 2} allowed us to use Levitan's wave method to compute local asymptotics for the spectral density `along invariant subspaces' and express the second local Weyl coefficients of $A$ as the sum of the second local Weyl coefficients of the $A_j$'s.

The same argument wouldn't work when approaching the problem via diagonalization. Indeed, let $\tilde{U}(t):=e^{-it\tilde A}$. Using \eqref{propagator relation 1},  \eqref{propagator relation 2} and \eqref{main theorem 1 equation 3}, it is not difficult to see that
\begin{equation}
\label{propagator relation 3}
B^*U^{(j)}(t)B=\tilde{P}^{(j)}e^{-ita_j}\mod C^{\infty}(\mathbb{R}; \Psi^{-\infty}_m)
\end{equation}
and
\begin{equation}
\label{propagator relation 4}
\theta(A)\,U(t)=B\theta(\tilde A)e^{-it\tilde A} B^*= B\left(\sum_{j=1}^{m^+} \tilde{P}^{(j)}e^{-ita_j}\right)B^* \mod C^{\infty}(\mathbb{R}; \Psi^{-\infty}_m),
\end{equation}
where $\tilde{P}^{(j)}:=\tilde v^{(j)}[\tilde v^{(j)}]^*$ (recall \eqref{eigenvectors tilde A prin}).  In \eqref{propagator relation 3} and \eqref{propagator relation 4}, unlike in \eqref{propagator relation 2}, the propagator of $A$ and that of $\tilde{A}$ are not related directly, but via conjugation by the pseudodifferential operator $B$.

\

All in all,  on the one hand the diagonalization approach presented here has the advantage of turning an elliptic system into a collection of elliptic scalar operators in a way that essentially preserves the spectrum and is compatible with pseudodifferential projections $P_j$, which may be quite useful when dealing with applications, say, to operators of interest in theoretical physics. On the other hand,  the approach pursued in \cite{part2} is, in a sense, more natural to study the spectrum of $A$, in that it captures finer information,  and it always works,  even in the presence of the topological obstructions from Remark~\ref{remark topological obstructions}.

\section{An example: the operator of linear elasticity}
\label{An example: the operator of linear elasticity}

In this section we apply our results to an explicit example: the Lam\'e operator in dimension 2. Throughout this section we adopt Einstein's summation convention over repeated indices.

Let $M$ be the 2-torus $\mathbb{T}^2$ endowed with the standard flat metric $g$.  The operator of linear elasticity $L$ on vector fields,  also known as the \emph{Lam\'e operator}, is defined in accordance with
\begin{equation*}
\label{definition of L variational form}
\frac12 \int_M g_{\alpha\beta}\,v^\alpha (Lv)^\beta \rho\, dx=E(v),
\end{equation*}
where $E(v)$ is the potential energy of elastic deformation
\begin{equation}
\label{potential energy of elastic deformation}
E(v):=\frac{1}{2}\int_M \left(\lambda(\nabla_\alpha v^\alpha)^2+\mu(\nabla_\alpha v_\beta+\nabla_\beta v_\alpha)\nabla^\alpha v^\beta \right) \rho\,dx,
\end{equation}
$\rho(x):=\sqrt{\operatorname{det}g_{\alpha\beta}(x)}$ is the Riemannian density,  $\nabla$ is the Levi-Civita connection, and the scalars $\lambda$ and $\mu$ are the \emph{Lam\'e parameters}. We refer the reader to \cite{MR0106584,diffeo} for further details on the derivation. In order to guarantee strong convexity, the Lam\'e parameters are assumed to satisfy the conditions
\begin{equation}
\label{strong convexity}
\mu>0, \qquad \lambda+\mu>0,
\end{equation}
see, e.g., \cite{miyanishi}. 
Integrating \eqref{potential energy of elastic deformation} by parts yields an explicit formula for $L$:
\begin{equation}
\label{L}
(Lv)^\alpha=-\mu(\nabla_\beta \nabla^\beta v^\alpha+\operatorname{Ric}^\alpha{}_\beta \,v^\beta)-(\lambda+\mu)\nabla^\alpha\nabla_\beta v^\beta.
\end{equation}
Here $\operatorname{Ric}$ is the Ricci tensor.

In order to apply our results, we need to turn the operator $L$, which acts on 2-vectors, into an operator acting on 2-columns of half-densities. This is done as follows.  

Let $e_j$, $j=1,2$, be a global orthonormal framing on $\mathbb{T}^2$; we denote by $e_j{}^\alpha$ the $\alpha$-th component of the $j$-th vector field.  Put
\begin{equation}
e^j{}_\alpha:=\delta^{jk}g_{\alpha\beta}\,e_k{}^\beta,
\end{equation}
and consider the operator $S$, 
\begin{equation}
(Sv)^j:=e^j{}_\alpha v^\alpha\,,
\end{equation}
which maps 2-vectors to 2-columns of scalar functions.
The operator of linear elasticity acting on half-densities is defined as
\begin{equation}
\label{L on half densities}
L_{1/2}:= \rho^{1/2}SLS^{-1}\rho^{-1/2}.
\end{equation}
Of course, $L_{1/2}\in \Psi^2_2$.

In what follows we will compute the subprincipal symbol of the approximate diagonalization of $L_{1/2}$.  In order to express our formulae in terms of geometric invariants, we need to introduce some additional definitions.

Let $\nabla_W$ be the Weitzenb\"ock connection associated with the framing $\{e_j\}_{j=1}^2$, namely, the curvature-free metric compatible connection whose connection coefficients are
\begin{equation*}
\Upsilon^{\alpha}{}_{\beta\gamma}:=e_j{}^\alpha\dfrac{\partial e^j{}_\gamma}{\partial x^\beta}.
\end{equation*}
As we are working in dimension two, the torsion tensor
$
T^{\alpha}{}_{\beta\gamma}:=\Upsilon^{\alpha}{}_{\beta\gamma}-\Upsilon^{\alpha}{}_{\gamma\beta}
$
of $\nabla_W$ is equivalent to a covector field
\begin{equation*}
\label{covector field t}
t_\alpha:=\frac12 T_\alpha{}^{\beta\gamma} \epsilon_{\beta\gamma}\, \rho,
\end{equation*}
where $\epsilon_{\beta\gamma}$ is the totally antisymmetric symbol, $\epsilon_{12}=+1$.
Here and further on, indices are raised and lowered using the metric, via the musical isomorphism.

By means of straightforward calculations, formulae \eqref{L}--\eqref{L on half densities} yield
\begin{equation}
\label{L principal}
(L_{1/2})_\prin=\mu h^2\,I+(\lambda+\mu)h^2\, pp^T
\end{equation}
and
\begin{equation}
\label{L subprincipal}
(L_{1/2})_\sub=i(\lambda+3\mu)t^\alpha\xi_\alpha \,\epsilon,
\end{equation}
where
\begin{equation}
\label{h and p}
h(x,\xi):=\sqrt{g^\alpha\beta(x)\xi_\alpha\xi_\beta}\,,
\qquad
p:=\frac{1}{h}\begin{pmatrix}
e_1{}^\alpha \xi_\alpha\\
e_2{}^\alpha \xi_\alpha
\end{pmatrix},
\qquad
\epsilon:=\begin{pmatrix}
0 & 1\\
-1 & 0
\end{pmatrix}.
\end{equation}

Analysing \eqref{L principal} we conclude that the eigenvalues of $(L_{1/2})_\prin$ are
\begin{equation}
\label{eigenvalue L prin}
h^{(1)}=\mu h^2, \qquad h^{(2)}=(\lambda+2\mu)h^2
\end{equation}
and the corresponding orthonormalised eigenvectors are
\begin{equation}
\label{eigenvectors L prin}
v^{(1)}=\epsilon \,p\,, \qquad v^{(2)}=p.
\end{equation}
Note that conditions \eqref{strong convexity} imply $h^{(2)}/h^{(1)}>1$, so, in particular,  that the eigenvalues are simple.

\begin{theorem}
\label{theorem subprincipal lame}
Let $\tilde{L}_{1/2}:=B^*L_{1/2}B$, where $B$ is the almost-unitary operator given by Theorems~\ref{main theorem 1} and~\ref{main theorem 3} for $A=L_{1/2}$. Then we have
\begin{equation}
\label{subprincipal L diagonal}
(\tilde{L}_{1/2})_\sub=0.
\end{equation}
\end{theorem}

\begin{proof}
Proving \eqref{subprincipal L diagonal} is equivalent to showing that
\begin{equation}
\label{calculation 0}
(B_j^*L_{1/2}B_j)_\sub=0, \qquad j=1,2.
\end{equation}

Let us begin by computing $(B_j)_\sub$, $j=1,2$.  It was shown in \cite{part1} that
\begin{equation}
\label{Pj sub lame}
(P_j)_\sub=0, \qquad j=1,2.
\end{equation}
Furthermore, it is not hard to see that
\begin{equation}
\label{calculation 1}
\{P^{(1)}, v^{(1)}\}=\epsilon\{P^{(2)}, v^{(2)}\}=0.
\end{equation}
On account of \eqref{Pj sub lame} and \eqref{calculation 1}, Theorem~\ref{main theorem 2} gives us
\begin{equation}
\label{Bj sub lame}
(B_j)_\sub=i f^{(j)}v^{(j)}, \qquad j=1,2.
\end{equation}
Of course, 
\begin{equation}
\label{Bj sub star lame}
(B_j^*)_\sub=[(B_j)_\sub]^*=-i f^{(j)}[v^{(j)}]^*.
\end{equation}

Now,  applying \eqref{subprincipal symbol of composition of three operators} to $B_j^*L_{1/2}B_j$ and using \eqref{Bj sub lame}, \eqref{Bj sub star lame} we obtain
\begin{equation}
\label{calculation 2}
\begin{split}
(B_j^*L_{1/2}B_j)_\sub
&
=
[v^{(j)}]^* (L_{1/2})_\sub v^{(j)}+\frac{i}2 \{[v^{(j)}]^*,  (L_{1/2})_\prin, v^{(j)}\}
\\
&
+\frac{i}2\left([v^{(j)}]^*\{(L_{1/2})_\prin, v^{(j)}\}+\{[v^{(j)}]^*,  (L_{1/2})_\prin\}v^{(j)}\right).
\end{split}
\end{equation}
Formulae \eqref{L subprincipal},  \eqref{h and p} and \eqref{eigenvectors L prin} imply
\begin{equation}
\label{calculation 3}
[v^{(j)}]^* (L_{1/2})_\sub v^{(j)}=0, \qquad j=1,2.
\end{equation}
The Spectral Theorem tells us that
\begin{equation*}
(L_{1/2})_\prin=h^2(\mu\,I+(\lambda+\mu)P^{(2)})=(\lambda+2\mu)I-(\lambda+\mu)P^{(1)}.
\end{equation*}
The latter, combined with the identities
\begin{equation*}
\{[v^{(j)}]^*,v^{(j)}\}=0, \qquad \{[v^{(j)}]^*,P^{(j)},v^{(j)}\}=0,\qquad j=1,2,
\end{equation*}
in turn gives us
\begin{equation}
\label{calculation 4}
\{[v^{(j)}]^*,  (L_{1/2})_\prin, v^{(j)}\}=0, \qquad j=1,2.
\end{equation}
Finally, we observe that
\begin{equation*}
\{(L_{1/2})_\prin,v^{(j)}\}=-\{[v^{(j)}]^*, (L_{1/2})_\prin\}^*
\end{equation*}
and
\begin{equation*}
\{[v^{(j)}]^*,  (L_{1/2})_\prin\}v^{(j)}=(\{[v^{(j)}]^*,  (L_{1/2})_\prin\}v^{(j)})^*,
\end{equation*}
so that
\begin{equation}
\label{calculation 5}
[v^{(j)}]^*\{(L_{1/2})_\prin, v^{(j)}\}+\{[v^{(j)}]^*,  (L_{1/2})_\prin\}v^{(j)}=0, \qquad j=1,2.
\end{equation}

Substituting \eqref{calculation 3}, \eqref{calculation 4} and \eqref{calculation 5} into \eqref{calculation 2} we arrive at \eqref{calculation 0}.
\end{proof}

Theorem~\ref{theorem subprincipal lame} implies that the second (global) Weyl coefficient for the operator of linear elasticity in dimension 2 vanishes, which agrees with known results \cite[Chapter~6]{SaVa}.

\section*{Acknowledgements}
\addcontentsline{toc}{section}{Acknowledgements}

I am grateful 
to Alex Strohmaier for raising questions that eventually led to this paper and for pointing out useful references;
to Jean-Claude Cuenin for helpful comments on a preliminary version of this manuscript and for bibliographic suggestions; to Grigori Rozenbloum for insightful discussions on the role of topological obstructions;
to Dmitri Vassiliev for numerous discussions on this and related topics throughout the years.

This work was supported by the Leverhulme Trust Research Project Grant RPG-2019-240, which is gratefully acknowledged.

\begin{appendices}

\section{Some useful results}
\label{Appendix A}

We collect in this appendix, for the reader's convenience, a few results from \cite{part2} that are used in the main text of the paper.

\begin{lemma}[{\cite[Lemma~3.6]{part2}}]
\label{lemma about bound on number of ev in interval}
Let $A\in\Psi^s_m$, $s>0$, be an operator as in Section~\ref{Statement of the problem}. Let
$
\mu_1\le\mu_2\le\dots\le \mu_r
$
be real numbers and let $u_k$, $k=1,\ldots,r$, be an orthonormal set in $L^2(M)$. 
Suppose that
\begin{equation*}
\label{lemma about bound on number of ev in interval equation 2}
\|(A-\mu_k)u_k\|_{L^2}\le \varepsilon, \qquad k=1,\ldots,r.
\end{equation*}
Then 
\begin{equation*}
\label{lemma about bound on number of ev in interval equation 3}
\# \{\lambda\ |\ \lambda\in \sigma(A)\cap [\mu_1-\sqrt{r}\,\varepsilon,\mu_r+\sqrt{r} \,\varepsilon] \} \ge  r.
\end{equation*}
\end{lemma}

\begin{lemma}[{\cite[Lemma~3.9]{part2}}]
\label{lemma about orthonormalisation}
Let $F$ be an Hermitian $r\times r$ matrix such that
\begin{equation}
\label{lemma about orthonormalisation equation 1}
\|F-I\|_{\max}\le \frac{1}{3r^2}\,,
\end{equation}
where $I$ is the $r\times r$ identity matrix and $\|F\|_{\max}:=\max_{1\le j,k \le r}|F_{jk}|$ is the max matrix norm. 
Then there exists an Hermitian matrix $G$ such that
\begin{equation}
\label{lemma about orthonormalisation equation 2}
GFG=I
\end{equation}
and
\begin{equation}
\label{lemma about orthonormalisation equation 3}
\|G-I\|_{\max}\le \|F-I\|_{\max}.
\end{equation}
\end{lemma}

Let $A\in\Psi^s_m$, $s>0$, be an operator as in Section~\ref{Statement of the problem}
and let
\begin{equation*}
\label{Weyl asymptotics for elliptic systems counting function definition}
N(\lambda):=
\begin{cases}
0 & \text{for}\ \lambda\le 0,\\
\sum_{k\,:\,0<\lambda_k<\lambda}1 &\text{for}\ \lambda> 0\\
\end{cases}
\end{equation*}
be its (positive) counting function.

\begin{theorem}[Weyl's law for elliptic systems {\cite[Theorem~B.1]{part2}}]
\label{Weyl asymptotics for elliptic systems theorem}
We have
\begin{equation*}
\label{Weyl asymptotics for elliptic systems theorem equation 1}
N(\lambda)=
b\lambda^{d/s}+o(\lambda^{d/s})
\quad\text{as}\quad\lambda\to+\infty,
\end{equation*}
where
\begin{equation*}
\label{Weyl asymptotics for elliptic systems theorem equation 2}
b=
\frac{1}{(2\pi)^d}\,\sum_{j=1}^{m^+}
\ \int\limits_{h^{(j)}(x,\xi)<1}\operatorname{dVol}_{T^*M}.
\end{equation*}
\end{theorem}

\end{appendices}

\end{document}